\documentclass{amsart}%
\usepackage{amsfonts}
\usepackage{amsmath}
\usepackage{amssymb}
\usepackage{graphicx}%
\usepackage{comment}
\providecommand{\U}[1]{\protect\rule{.1in}{.1in}}
\newtheorem{theorem}{Theorem}
\theoremstyle{plain}

\newtheorem{corollary}{Corollary}

\newtheorem{lemma}{Lemma}

\newtheorem{remark}{Remark}

\numberwithin{equation}{section}
\usepackage{url}
\begin{document}
\title[Irrationality exponents of generalized Hone series]{Irrationality exponents of generalized Hone series}
\author{Daniel Duverney, Takeshi Kurosawa and Iekata Shiokawa}
\address{ }
\email{ }
\date{Dec. 2019}
\subjclass{ }

\begin{abstract}
We compute the exact irrationality exponents of certain series of rational
numbers, first studied in a special case by Hone, by transforming them into
suitable continued fractions.

\end{abstract}
\maketitle

\section{Main theorem}

For a real number $\alpha,$ the irrationality exponent $\mu\left(
\alpha\right)  $ is defined by the infimum of the set of numbers $\mu$ for
which the inequality%
\begin{equation}
\left\vert \alpha-\frac{p}{q}\right\vert <\frac{1}{q^{\mu}} \label{1.1}%
\end{equation}
has only finitely many rational solutions $p/q$, or equivalently the supremum
of the set of numbers $\mu$ for which the inequality (\ref{1.1}) has
infinitely many solutions. If $\alpha$ is irrational, then $\mu\left(
\alpha\right)  \geq2$. If $\alpha$ is a real algebraic irrationality, then
$\mu\left(  \alpha\right)  =2$ by Roth's theorem \cite{Roth}. If $\mu\left(
\alpha\right)  =\infty,$ then $\alpha$ is called a Liouville number.

For every sequence $\left(  u_{n}\right)  _{n\geq1}$ of nonzero numbers or
indeterminates, we define $u_{0}=1$ and%
\begin{equation}
\theta u_{k}=\frac{u_{k+1}}{u_{k}},\quad\theta^{2}u_{k}=\theta\left(  \theta
u_{k}\right)  =\frac{u_{k+2}u_{k}}{u_{k+1}^{2}}\quad\left(  k\geq0\right)  .
\label{fn}%
\end{equation}

\begin{theorem}
\label{ThA}Let $\left(  x_{n}\right)  _{n\geq1}$ be an increasing sequence of
integers with $x_{1}\geq2$ and $\left(  y_{n}\right)  _{n\geq1}$ be a sequence
of nonzero integers such that $x_{1}>y_{1}\geq1$%
\begin{equation}
\frac{\theta^{2}x_{n}-\theta^{2}y_{n}}{x_{n}}\in\mathbb{Z}_{>0}\quad\left(
n\geq0\right)  . \label{C1}%
\end{equation}
Assume that%
\begin{enumerate}
\renewcommand{\labelenumi}{(\roman{enumi})}
\let\theenumi\labelenumi
\item \label{item1}
\[
\log\left\vert y_{n+2}\right\vert =o\left(  \log
x_{n}\right),
\]
\item \label{item2}
\[
\liminf_{n\rightarrow\infty}\frac{\log x_{n+1}}{\log
x_{n}}>2.
\]
\end{enumerate}
Then the series%
\begin{equation}
\sigma=\sum_{n=1}^{\infty}\frac{y_{n}}{x_{n}} \label{C2}%
\end{equation}
is convergent and%
\[
\mu\left(  \sigma\right)  =\max\left\{  \limsup_{n\rightarrow\infty}\frac{\log
x_{n+1}}{\log x_{n}}\text{ },\text{ }2+\frac{1}{\liminf\limits_{n\rightarrow
\infty}\dfrac{\log x_{n+1}}{\log x_{n}}-1}\right\}  .
\]
\end{theorem}

\begin{remark}
\label{Rem1}The assumption \ref{item2} implies that%
\[
2+\frac{1}{\liminf\limits_{n\rightarrow\infty}\dfrac{\log x_{n+1}}{\log x_{n}%
}-1}<3.
\]
Moreover, if the limit $\lambda:=\lim_{n\rightarrow\infty}\left(  \log
x_{n+1}/\log x_{n}\right)  $ exists, then%
\[
\mu\left(  \sigma\right)  =\max\left\{  \lambda\text{ },\text{ }2+\frac
{1}{\lambda-1}\right\}  ,
\]
and so%
\begin{equation}
\mu\left(  \sigma\right)  =\left\{
\begin{array}
[c]{l}%
\lambda\quad\text{if\quad}\lambda\geq\dfrac{3+\sqrt{5}}{2}=2.618...,\\
2+\dfrac{1}{\lambda-1}\quad\text{if\quad}2<\lambda<\dfrac{3+\sqrt{5}}{2}.
\end{array}
\right.  \text{ } \label{MuSigma}%
\end{equation}
Hence, under the hypotheses of Theorem \ref{ThA}, $\mu\left(  \sigma\right)
>2$ and therefore $\sigma$ is transcendental.
\end{remark}

Examples of series $\sigma$ satisfying the assumptions of Theorem \ref{ThA}
have been first given by Hone \cite{H1} in the case where $y_{n}=1$ for every
positive integer $n,$ and later by Varona \cite{Va} in the case where
$y_{n}=\left(  -1\right)  ^{n}.$ Both Hone and Varona computed the expansion
in regular continued fraction of $\sigma$ in these special cases and succeeded
in proving its transcendence by using Roth's theorem. For more expansions in
regular continued fraction, see also \cite{H2}, \cite{H3}, and \cite{H4}.

In this paper, we will use basically the same method and transform $\sigma$
given by (\ref{C2}) into a continued fraction (not regular in general) by
using Lemma \ref{Lem2} in Section \ref{sec:lemma}. Then we will reach our conclusion by
applying a formula which gives the irrationality exponent of continued
fractions under convenient assumptions (Lemma \ref{Lem3}, also in Section \ref{sec:lemma}).

The paper is organized as follows. In Section \ref{sec:app}, we will give examples of
series generalizing both Hone and Varona series, and show how Theorem
\ref{ThA} applies to these series (see formula (\ref{Ap}) below). In Section \ref{sec:lemma}, 
we will state three lemmas which will be useful in the proof of theorem
\ref{ThA}. The proof of Theorem \ref{ThA} will be given in Section \ref{sec:proof}. Finally,
in Section \ref{sec:asym} we will give the proof of the asymptotic estimates used in
Section \ref{sec:app}.

\section{Applications}\label{sec:app}

For any nonzero integer $a$ and non-constant $P(X) \in {\mathbb Z}_{\geq 0}[X]$ with $P(0)=0$, 
we define the sequence $\{y_n\}$ by 
\[
y_{n}=a^{P(n)}\qquad(n\geq 0).
\]
We see that 
$$\theta^2 y_n = a^{P(n+2)-2P(n+1)+P(n)}\in {\mathbb Z}_{>0},$$
since $P(n+2)-2P(n+1)+P(n)\equiv P(n+2)+P(n) \equiv 2P(n)\equiv 0~(\mbox{mod}~2)$. 
We define the sequence $\{x_n\}$ by the recurrence relation

%In all this section, let
%\[
%P(X)=\sum_{i=0}^{p}\alpha_{i}X^{i}%
%\]
%be any non-constant polynomial with $\alpha_{i}\in\mathbb{Z}_{\geq0}$ for
%every $i=0,1,\ldots,p$ and $\alpha_{p}\neq0,$ and let also%
\begin{equation}
x_{n+2}x_{n}=x_{n+1}^{2}\left(  x_{n}Q\left(  x_{n},x_{n+1}\right)
+ \theta^2 y_n \right)\label{RecX}%
\end{equation}
with the initial conditions
\begin{equation}
x_{0}=1,\quad x_{1}\in\mathbb{Z}_{>1},\label{Init}%
\end{equation}
where
\begin{equation}
Q(X,Y)=\sum_{i=0}^{q}\sum_{j=0}^{r}\beta_{i,j}X^{i}Y^{j} \in\mathbb{Z}_{\geq0}[X,Y],\quad \beta_{q,r}\neq0. \label{Pol1}%
\end{equation}

\begin{comment}
be any non-constant polynomial with $\beta_{i,j}\in\mathbb{Z}_{\geq0}$ for
every $i=0,1,\ldots,q$ and $j=0,1,\ldots,r,$ and $\beta_{q,r}\neq0.$

\begin{lemma}
\label{LemA}Let $\Delta$ be the difference operator, defined for every
sequence $\left(  u_{n}\right)  _{n\geq0}$ of real numbers by%
\[
\Delta u_{n}=u_{n+1}-u_{n}.
\]
Then $\Delta^{2}P(n)$ is a non-negative even integer for every positive
integer $n.$
\end{lemma}

\begin{proof}
By linearity, it is sufficient to prove the result when $P(n)=n^{m},$ with
$m\in\mathbb{Z}_{>0}.$ In this case, we have%
\[
\Delta^{2}n^{m}=\Delta\left(  \left(  n+1\right)  ^{m}-n^{m}\right)
=\sum_{i=1}^{m-1}\binom{m}{i}\left(  \left(  n+1\right)  ^{i}-n^{i}\right)  .
\]
Furthermore, since $\left(  n+1\right)  ^{i}-n^{i}\equiv1$ $\left(
\operatorname{mod}2\right)  $ for all $i\geq1,$ we find%
\[
\Delta^{2}n^{m}\equiv\sum_{i=1}^{m-1}\binom{m}{i}=2^{m}-2\equiv0\quad\left(
\operatorname{mod}2\right)  ,
\]
which proves Lemma \ref{LemA}.
\end{proof}

With the above notation, we first define $y_{n}$ by
\[
y_{n}=a^{P(n)},
\]
where $a$ is a given nonzero rational integer. Then by Lemma \ref{LemA}%
\[
\theta^{2}y_{n}=a^{\Delta^{2}P(n)}\in\mathbb{Z}_{>0}.
\]
\end{comment}

It is clear that $x_{n}>0$ for every $n\geq0.$ Besides, an easy induction
shows that%
\begin{equation}
x_{n}\geq\left(  x_{1}\right)  ^{2^{n-1}}\label{Min}%
\end{equation}
and that $x_{n-1}$ divides $x_{n}$ for every $n\geq1.$ 
By (\ref{Min}), the series $\sum_{n=1}^\infty y_n/x_n$ is convergent. 
Moreover we will prove
in Section \ref{sec:asym}, Corollary \ref{Cor3}, that%
\[
\log x_{n}{\sim}C\lambda^{n}\qquad \mbox{as}\qquad n \rightarrow \infty,
\]
where $C$ is a positive constant and
\[
\lambda=\lim_{n\rightarrow\infty}\frac{\log x_{n+1}}{\log x_{n}}=\frac{1}%
{2}\left(  r+2+\sqrt{\left(  r+2\right)  ^{2}+4q}\right)  .
\]
As $Q$ is non-constant, we have $q+r\neq0$ and%
\begin{align*}
\lambda &  =1+\sqrt{2}\quad\text{if}\quad q=1\text{ and }r=0,\\
\lambda &  \geq\frac{3+\sqrt{5}}{2}\quad\text{otherwise.}%
\end{align*}
Therefore, applying Theorem \ref{ThA} and (\ref{MuSigma}), we obtain

\begin{theorem}
\label{Th2New}Let $(x_{n})_{n\geq1}$ be as above. Define the number $\sigma$
by%
\[
\sigma=\sum_{n=1}^{\infty}\frac{y_{n}}{x_{n}}=\sum_{n=1}^{\infty}%
\frac{a^{P(n)}}{x_{n}}.
\]
Then we have%
\begin{equation}
\mu\left(  \sigma\right)  =\left\{
\begin{array}
[c]{l}%
2+\frac{1}{\sqrt{2}}\quad\text{if}\quad q=1\text{ and }r=0\\
\lambda\quad\text{otherwise.}%
\end{array}
\right.  \text{ } \label{Ap}%
\end{equation}

\end{theorem}

\begin{remark}
\label{Rem11}Hone and Varona series in \cite{H1} and \cite{Va} are obtained as
special cases of (\ref{RecX}) by taking $P(X)=X$ with $a=1$ and $a=-1$ respectively.
\end{remark}

\section{Lemmas}\label{sec:lemma}

In this section, we prepare some lemmas for the proof of Theorem \ref{ThA}.

\begin{lemma}
\label{Lem1}Let $(x_{n})_{n\geq1}$ and $(y_{n})_{n\geq1}$ be sequences as in
Theorem \ref{ThA}. We have%
\begin{equation}
\log\left\vert y_{n+1}\right\vert =o\left(  \log x_{n}\right)  ,\quad
\log\left\vert y_{n}\right\vert =o\left(  \log x_{n}\right)  , \label{C3}%
\end{equation}%
\begin{equation}
\sum_{j=1}^{n}\log\left\vert y_{j}\right\vert =o\left(  \log x_{n}\right)  .
\label{C4}%
\end{equation}

\end{lemma}

\begin{proof}
The assumption \ref{item2} implies for large $n$ that%
\begin{align*}
\frac{\log\left\vert y_{n+1}\right\vert }{\log x_{n}}  &  =\frac
{\log\left\vert y_{n+1}\right\vert }{\log x_{n-1}}\frac{\log x_{n-1}}{\log
x_{n}}\leq\frac{\log\left\vert y_{n+1}\right\vert }{2\log x_{n-1}},\\
\frac{\log\left\vert y_{n}\right\vert }{\log x_{n}}  &  =\frac{\log\left\vert
y_{n}\right\vert }{\log x_{n-1}}\frac{\log x_{n-1}}{\log x_{n}}\leq\frac
{\log\left\vert y_{n}\right\vert }{2\log x_{n-1}},
\end{align*}
which proves (\ref{C3}) by using \ref{item1}. Now by (\ref{C3}) we have%
\[
\sum_{j=1}^{n}\log\left\vert y_{j}\right\vert =o\left(  \sum_{j=1}^{n}\log
x_{j}\right)
\]
and by \ref{item2} there exists a constant $K>0$ such that%
\[
\log x_{j}\leq\frac{K}{2^{n-j}}\log x_{n}\quad\left(  1\leq j\leq n\right)  ,
\]
which proves (\ref{C4}).
\end{proof}

\begin{lemma}[{\cite[Theorem 2]{DKS}}]
\label{Lem2}  Let $x_{1},$ $x_{2},$ $\ldots,$ $y_{1},$
$y_{2},$ $\ldots$ be indeterminates. Then for every $n\geq1,$%
\begin{equation}
\sigma_{n}=\sum_{k=1}^{n}\frac{y_{k}}{x_{k}}=\frac{a_{1}}{b_{1}}%
%TCIMACRO{\QATOP{{}}{+}}%
%BeginExpansion
\genfrac{}{}{0pt}{}{{}}{+}%
%EndExpansion
\frac{a_{2}}{b_{2}}%
%TCIMACRO{\QATOP{{}}{+\cdots}}%
%BeginExpansion
\genfrac{}{}{0pt}{}{{}}{+\cdots}%
%EndExpansion%
%TCIMACRO{\QATOP{{}}{+}}%
%BeginExpansion
\genfrac{}{}{0pt}{}{{}}{+}%
%EndExpansion
\dfrac{a_{2n}}{b_{2n}}, \label{Hone3}%
\end{equation}
where $a_{1}=y_{1},$ $b_{1}=x_{1}-y_{1},$ and for $k\geq1$%
\begin{align*}
a_{2k}  &  =\theta y_{k-1},\quad a_{2k+1}=\theta^{2}y_{k-1},\\
\quad b_{2k}  &  =x_{k-1},\quad b_{2k+1}=\frac{\theta^{2}x_{k-1}-\theta
^{2}y_{k-1}}{x_{k-1}}.
\end{align*}

\end{lemma}

\begin{lemma}[{\cite[Corollary 4]{DS}}]
\label{Lem3} Let an infinite continued fraction%
\[
\alpha=\frac{a_{1}}{b_{1}}%
%TCIMACRO{\QATOP{{}}{+}}%
%BeginExpansion
\genfrac{}{}{0pt}{}{{}}{+}%
%EndExpansion
\frac{a_{2}}{b_{2}}%
%TCIMACRO{\QATOP{{}}{+}}%
%BeginExpansion
\genfrac{}{}{0pt}{}{{}}{+}%
%EndExpansion
\dfrac{a_{3}}{b_{3}}%
%TCIMACRO{\QATOP{{}}{+}}%
%BeginExpansion
\genfrac{}{}{0pt}{}{{}}{+}%
%EndExpansion
\cdots
\]
be convergent, where $a_{n},$\ $b_{n}$\ $\left(  n\geq1\right)  $\ are
nonzero rational integers. Assume that%
\begin{enumerate}
\renewcommand{\labelenumi}{(\Roman{enumi})}
\let\theenumi\labelenumi
\item\label{itemI}
\[
\sum_{n=1}^{+\infty}\left\vert \frac{a_{n+1}}{b_{n}b_{n+1}%
}\right\vert <\infty,
\]
\item\label{itemII}
\[
\text{\quad}\lim_{n\rightarrow+\infty}\frac{\log\left\vert a_{n}%
\right\vert }{\log\left\vert b_{n}\right\vert }=0.
\]
\end{enumerate}
Then $\alpha$\ is irrational and
\begin{equation}
\mu\left(  \alpha\right)  =2+\limsup_{n\rightarrow+\infty}\frac{\log\left\vert
b_{n+1}\right\vert }{\log\left\vert b_{1}b_{2}\cdots b_{n}\right\vert }.
\label{C9}%
\end{equation}
\end{lemma}

\section{Proof of Theorem \ref{ThA}}\label{sec:proof}

The assumptions \ref{item1}, \ref{item2} and (\ref{C3}) imply that
\begin{equation}
\max\left(  \left\vert y_{k}\right\vert ,\left\vert y_{k+1}\right\vert
,\left\vert y_{k+2}\right\vert \right)  \leq x_{k}^{\varepsilon}\label{Major}%
\end{equation}
and that $x_{k+1}\geq x_{k}^{2}$ $\left(  k\geq k_{0}\left(  \varepsilon
\right)  \right)  $ for any $\varepsilon\in (0,1)$. We have%
\[
\sum_{k=k_{0}}^{\infty}\frac{\left\vert y_{k}\right\vert }{x_{k}}\leq
\sum_{k=k_{0}}^{\infty}\frac{1}{x_{k}^{1-\varepsilon}}\leq\sum_{k=0}^{\infty
}\frac{1}{x_{k_{0}}^{\left(  1-\varepsilon\right)  2^{k}}},
\]
and hence the series (\ref{C2}) is absolutely convergent. Using Lemma
\ref{Lem2}, we get the continued fraction expansion of $\sigma=\lim
_{n\rightarrow\infty}\sigma_{n}.$ To apply Lemma \ref{Lem3} we transform this
to a continued fraction with integral partial numerators and denominators by
using the formula%
\[
\frac{a_{1}}{b_{1}}%
%TCIMACRO{\QATOP{{}}{+}}%
%BeginExpansion
\genfrac{}{}{0pt}{}{{}}{+}%
%EndExpansion
\frac{a_{2}}{b_{2}}%
%TCIMACRO{\QATOP{{}}{+}}%
%BeginExpansion
\genfrac{}{}{0pt}{}{{}}{+}%
%EndExpansion
\dfrac{a_{3}}{b_{3}}%
%TCIMACRO{\QATOP{{}}{+}}%
%BeginExpansion
\genfrac{}{}{0pt}{}{{}}{+}%
%EndExpansion
\cdots=\frac{r_{1}a_{1}}{r_{1}b_{1}}%
%TCIMACRO{\QATOP{{}}{+}}%
%BeginExpansion
\genfrac{}{}{0pt}{}{{}}{+}%
%EndExpansion
\frac{r_{2}r_{1}a_{2}}{r_{2}b_{2}}%
%TCIMACRO{\QATOP{{}}{+}}%
%BeginExpansion
\genfrac{}{}{0pt}{}{{}}{+}%
%EndExpansion
\dfrac{r_{2}r_{3}a_{3}}{r_{3}b_{3}}%
%TCIMACRO{\QATOP{{}}{+}}%
%BeginExpansion
\genfrac{}{}{0pt}{}{{}}{+}%
%EndExpansion
\cdots.
\]
By taking $r_{2k}=1$ and $r_{2k+1}=y_{k}^{2}$ $(k\geq1),$ we obtain the
expansion%
\[
\sigma=\frac{a_{1}}{b_{1}}%
%TCIMACRO{\QATOP{{}}{+}}%
%BeginExpansion
\genfrac{}{}{0pt}{}{{}}{+}%
%EndExpansion
\frac{a_{2}}{b_{2}}%
%TCIMACRO{\QATOP{{}}{+\cdots}}%
%BeginExpansion
\genfrac{}{}{0pt}{}{{}}{+\cdots}%
%EndExpansion%
%TCIMACRO{\QATOP{{}}{+}}%
%BeginExpansion
\genfrac{}{}{0pt}{}{{}}{+}%
%EndExpansion
\dfrac{a_{n}}{b_{n}}%
%TCIMACRO{\QATOP{{}}{+\cdots}}%
%BeginExpansion
\genfrac{}{}{0pt}{}{{}}{+\cdots}%
%EndExpansion
,
\]
where $a_{1}=y_{1},$ $b_{1}=x_{1}-y_{1}\geq1,$ and for $k\geq1$%
\begin{align}
a_{2k} &  =y_{k-1}y_{k},\quad a_{2k+1}=y_{k-1}y_{k+1},\label{C11}\\
\quad b_{2k} &  =x_{k-1},\quad b_{2k+1}=y_{k}^{2}\frac{\theta^{2}%
x_{k-1}-\theta^{2}y_{k-1}}{x_{k-1}}\geq1.\label{C12}%
\end{align}
We certify the conditions \ref{itemI} and \ref{itemII} in
Lemma \ref{Lem3}. First for \ref{itemI}, we have%
\[
\sum_{n=1}^{+\infty}\left\vert \frac{a_{n+1}}{b_{n}b_{n+1}}\right\vert
\leq\sum_{k=1}^{+\infty}\left\vert \frac{a_{2k+1}}{b_{2k}b_{2k+1}}\right\vert
+\sum_{k=1}^{+\infty}\left\vert \frac{a_{2k}}{b_{2k-1}b_{2k}}\right\vert ,
\]
where%
\begin{align*}
\left\vert \frac{a_{2k+1}}{b_{2k}b_{2k+1}}\right\vert  &  \leq\frac{\left\vert
y_{k-1}y_{k+1}\right\vert }{x_{k-1}y_{k}^{2}}\leq\frac{\left\vert
y_{k-1}y_{k+1}\right\vert }{x_{k-1}},\\
\left\vert \frac{a_{2k}}{b_{2k-1}b_{2k}}\right\vert  &  \leq\frac{\left\vert
y_{k-1}y_{k}\right\vert }{y_{k-1}^{2}x_{k-1}}\leq\frac{\left\vert y_{k-1}%
y_{k}\right\vert }{x_{k-1}},
\end{align*}
noting that $b_{2k+1}\geq y_{k}^{2}$ by (\ref{C1}). By using (\ref{Major}) we
deduce similarly as above that%
\[
\max\left\{  \sum_{k=k_{0}+1}^{+\infty}\frac{\left\vert y_{k-1}y_{k+1}%
\right\vert }{x_{k-1}},\sum_{k=k_{0}+1}^{+\infty}\frac{\left\vert y_{k-1}%
y_{k}\right\vert }{x_{k-1}}\right\}  \leq\sum_{k=k_{0}+1}^{+\infty}\frac
{1}{x_{k-1}^{1-2\varepsilon}}%
\]
and \ref{itemI} follows. Now we prove that \ref{itemII} holds. We have by \ref{item1} and
(\ref{C3})%
\[
\lim_{k\rightarrow\infty}\frac{\log\left\vert a_{2k}\right\vert }{\log b_{2k}%
}=\lim_{k\rightarrow\infty}\frac{\log\left\vert y_{k-1}\right\vert
+\log\left\vert y_{k}\right\vert }{\log x_{k-1}}=0
\]
and%
\begin{equation}
\lim_{k\rightarrow\infty}\frac{\log\left\vert a_{2k+1}\right\vert }{\log
b_{2k+1}}=\lim_{k\rightarrow\infty}\frac{\log\left\vert y_{k-1}\right\vert
+\log\left\vert y_{k+1}\right\vert }{2\log\left\vert y_{k}\right\vert
+\log\left(  \theta^{2}x_{k-1}-\theta^{2}y_{k-1}\right)  -\log x_{k-1}%
}.\label{C10}%
\end{equation}
Here, since $\theta^{2}x_{k-1}-\theta^{2}y_{k-1}\geq x_{k-1}$ by (\ref{C1}),
we have%
\begin{equation}
\frac{\left\vert \theta^{2}y_{k-1}\right\vert }{\theta^{2}x_{k-1}}\leq
\frac{\left\vert \theta^{2}y_{k-1}\right\vert }{x_{k-1}+\theta^{2}y_{k-1}%
}.\label{C20}%
\end{equation}
However, by (\ref{C3}) we can write for every $\varepsilon\in (0,1)$%
\begin{align*}
\log\frac{\left\vert \theta^{2}y_{k-1}\right\vert }{x_{k-1}} &  =\log
\left\vert y_{k+1}\right\vert +\log\left\vert y_{k-1}\right\vert
-2\log\left\vert y_{k}\right\vert -\log x_{k-1}\\
&  =-\log x_{k-1}+o\left(  \log x_{k-1}\right)  \leq-\varepsilon\log
x_{k-1}\quad\left(  k\rightarrow\infty\right)  ,
\end{align*}
which implies%
\[
\frac{\left\vert \theta^{2}y_{k-1}\right\vert }{x_{k-1}}\leq\frac{1}%
{x_{k-1}^{\varepsilon}}\quad\left(  k\rightarrow\infty\right)  .
\]
Therefore from (\ref{C20}) we see that for every $\varepsilon\in (0,1)$%
\begin{equation}
\frac{\left\vert \theta^{2}y_{k-1}\right\vert }{\theta^{2}x_{k-1}}\leq
\frac{\left\vert \theta^{2}y_{k-1}\right\vert }{x_{k-1}+\theta^{2}y_{k-1}}%
\leq\frac{2}{x_{k-1}^{\varepsilon}}\rightarrow0\quad\left(  k\rightarrow
\infty\right)  ,\label{C21}%
\end{equation}
which yields%
\[
\log\left(  \theta^{2}x_{k-1}-\theta^{2}y_{k-1}\right)  =\log\theta^{2}%
x_{k-1}+o(1).
\]
Hence we get from (\ref{C10}) and (\ref{C3})%
\[
\lim_{k\rightarrow\infty}\frac{\log\left\vert a_{2k+1}\right\vert }{\log
b_{2k+1}}=\lim_{k\rightarrow\infty}\frac{\log\left\vert y_{k-1}\right\vert
+\log\left\vert y_{k+1}\right\vert }{2\log\left\vert y_{k}\right\vert +\log
x_{k+1}-2\log x_{k}+o(1)}=0,
\]
and \ref{itemII} is ensured. Now we compute the right-hand side of (\ref{C9}). We
have by (\ref{C12})%
\begin{align*}
b_{1}b_{2}\cdots b_{2k+1} &  =\left(  x_{1}-y_{1}\right)  \prod_{j=0}%
^{k-1}y_{j+1}^{2}\left(  \theta^{2}x_{j}-\theta^{2}y_{j}\right)  \\
&  =\left(  x_{1}-y_{1}\right)  \theta^{2}x_{0}\theta^{2}x_{1}\cdots\theta
^{2}x_{k-1}\prod_{j=0}^{k-1}y_{j+1}^{2}\left(  1-\frac{\theta^{2}y_{j}}%
{\theta^{2}x_{j}}\right)  \\
&  =\frac{x_{1}-y_{1}}{x_{1}}\frac{x_{k+1}}{x_{k}}\prod_{j=0}^{k-1}y_{j+1}%
^{2}\left(  1-\frac{\theta^{2}y_{j}}{\theta^{2}x_{j}}\right)  .
\end{align*}
Now the infinite product%
\[
\prod_{j=0}^{\infty}\left(  1-\frac{\theta^{2}y_{j}}{\theta^{2}x_{j}}\right)
\]
is convergent by (\ref{C21}) and (\ref{Min}). Hence we get by (\ref{C4})%
\begin{align}
\log\left(  b_{1}b_{2}\cdots b_{2k+1}\right)   &  =\log x_{k+1}-\log
x_{k}+2\sum_{j=1}^{k}\log\left\vert y_{j}\right\vert +O(1)\nonumber\\
&  =\log x_{k+1}-\log x_{k}+o\left(  \log x_{k}\right)  .\label{C24}%
\end{align}
Therefore%
\begin{equation}
\limsup_{k\rightarrow+\infty}\frac{\log b_{2k+2}}{\log\left(  b_{1}b_{2}\cdots
b_{2k+1}\right)  }=\limsup_{k\rightarrow+\infty}\frac{\log x_{k}}{\log
x_{k+1}-\log x_{k}}.\label{C25}%
\end{equation}
Furthermore, we have by (\ref{C24})%
\[
\log\left(  b_{1}b_{2}\cdots b_{2k}\right)  =\log\left(  b_{1}b_{2}\cdots
b_{2k-1}\right)  +\log b_{2k}=\log x_{k}+o\left(  \log x_{k}\right)
\]
and%
\[
\log b_{2k+1}=\log\left(  b_{1}b_{2}\cdots b_{2k+1}\right)  -\log\left(
b_{1}b_{2}\cdots b_{2k}\right)  =\log x_{k+1}-2\log x_{k}+o\left(  \log
x_{k}\right)  .
\]
Hence%
\begin{equation}
\limsup_{k\rightarrow+\infty}\frac{\log b_{2k+1}}{\log\left(  b_{1}b_{2}\cdots
b_{2k}\right)  }=\limsup_{n\rightarrow+\infty}\frac{\log x_{k+1}}{\log x_{k}%
}-2.\label{C26}%
\end{equation}
Therefore, it follows from (\ref{C9}), (\ref{C25}) and (\ref{C26}) that%
\[
\mu\left(  \sigma\right)  =\max\left\{  2+\frac{1}{\liminf
\limits_{k\rightarrow\infty}\dfrac{\log x_{k+1}}{\log x_{k}}-1}\text{ },\text{
}\limsup_{k\rightarrow\infty}\frac{\log x_{k+1}}{\log x_{k}}\right\}
\]
and the proof of Theorem \ref{ThA} is completed.

\section{Asymptotic behaviour}\label{sec:asym}

We now study the asymptotic behaviour of sequences $\left(  x_{n}\right)
_{n\geq1}$ satisfying (\ref{Init}) and (\ref{RecX}). We follow basically the
method indicated in \cite{H1}. Let $\left(  u_{n}\right)  _{n\geq0}$ be any
sequence of complex numbers satisfying the recurrence relation%
\begin{equation}
u_{n+2}-Au_{n+1}+Bu_{n}=\tau_{n}\quad\left(  n\geq0\right)  , \label{As1}%
\end{equation}
where $A$ and $B$ are complex numbers with $A^{2}-4B\neq0$ and $\tau_{n}$ is a
function of $n,$ $u_{n}$ and $u_{n+1}$. As $A^{2}-4B\neq0,$ the equation%
\begin{equation}
x^{2}-Ax+B=0. \label{Eq}%
\end{equation}
has two distinct roots $\lambda$ and $\nu,$ with $\lambda\neq\nu.$ Morever, at
least one of these roots is not zero, and we can assume without loss of
generality that $\lambda\neq0.$

\begin{theorem}
\label{Expr}Assume that $\left(  u_{n}\right)  _{n\geq0}$ satisfies
(\ref{As1})$.$ Let $\lambda$ and $\nu,$ with $\lambda\neq\nu$ and $\lambda
\neq0,$ be the roots of (\ref{Eq}). Then for every $n\geq1$%
\begin{align}
u_{n}  &  =\frac{1}{\lambda-\nu}\left(  \left(  u_{1}-\nu u_{0}\right)
\lambda^{n}+\left(  \lambda u_{0}-u_{1}\right)  \nu^{n}+\sum_{k=1}^{n-1}%
\tau_{k-1}\left(  \lambda^{n-k}-\nu^{n-k}\right)  \right)  \quad\left(
B\neq0\right)  ,\label{As3}\\
u_{n}  &  =\frac{1}{\lambda}\left(  u_{1}\lambda^{n}+\sum_{k=1}^{n-1}%
\tau_{k-1}\lambda^{n-k}\right)  \quad\left(  B=0\right)  . \label{As4}%
\end{align}

\end{theorem}

\begin{proof}
First we assume that $B\neq0,$ which implies $\nu\neq0.$ For every $n\geq0,$
let%
\[
v_{n}=\sum_{k=1}^{n}\tau_{k-1}\left(  \lambda^{n-k}-\nu^{n-k}\right)  .
\]
We have for $n\geq0$%
\begin{align*}
v_{n+1}  &  =\sum_{k=1}^{n+1}\tau_{k-1}\left(  \lambda^{n+1-k}-\nu
^{n+1-k}\right)  =\sum_{k=1}^{n}\tau_{k-1}\left(  \lambda^{n-k+1}-\nu
^{n-k+1}\right)  ,\\
v_{n+2}  &  =\sum_{k=1}^{n+1}\tau_{k-1}\left(  \lambda^{n-k+2}-\nu
^{n-k+2}\right)  =\left(  \lambda-\nu\right)  \tau_{n}+\sum_{k=1}^{n}%
\tau_{k-1}\left(  \lambda^{n-k+2}-\nu^{n-k+2}\right)  .
\end{align*}
Therefore, for every $n\geq0,$ the sequence%
\[
w_{n}=\frac{1}{\lambda-\nu}\left(  \left(  u_{1}-\nu u_{0}\right)  \lambda
^{n}+\left(  \lambda u_{0}-u_{1}\right)  \nu^{n}+v_{n}\right)
\]
satisfies $w_{n+2}-Aw_{n+1}+Bw_{n}=\tau_{n}.$ Moreover $w_{0}=u_{0}$ and
$w_{1}=u_{1}$ since $v_{0}=v_{1}=0.$ Hence $w_{n}=u_{n}$ for every $n\geq0$,
which proves Theorem \ref{Expr} when $B\neq0$ since%
\[
v_{n}=\sum_{k=1}^{n-1}\tau_{k-1}\left(  \lambda^{n-k}-\nu^{n-k}\right)
\quad\left(  n\geq1\right)  .
\]
Letting $B\rightarrow0$ in (\ref{As3}), we obtain (\ref{As4}).
\end{proof}

\begin{corollary}
\label{Corol}With the notations of Theorem \ref{Expr}, assume that $\left\vert
\nu\right\vert <\left\vert \lambda\right\vert $ and that $\left\vert
\lambda\right\vert >1.$ Assume moreover that $\tau_{n}$ is bounded. Then
\begin{align}
u_{n}  &  =C\lambda^{n}+O\left(  1\right)  \quad\text{if}\quad\left\vert
\nu\right\vert <1,\label{un1}\\
u_{n}  &  =C\lambda^{n}+O\left(  n\right)  \quad\text{if}\quad\left\vert
\nu\right\vert =1,\label{un2}\\
u_{n}  &  =C\lambda^{n}+D\nu^{n}+O\left(  1\right)  \quad\text{if}%
\quad\left\vert \nu\right\vert >1, \label{un3}%
\end{align}
where%
\[
C=\frac{1}{\lambda-\nu}\left(  u_{1}-\nu u_{0}+\sum_{k=1}^{\infty}\tau
_{k-1}\lambda^{-k}\right)
\]
and, in the case where $\left\vert \nu\right\vert >1,$%
\[
D=\frac{1}{\lambda-\nu}\left(  \lambda u_{0}-u_{1}-\sum_{k=1}^{\infty}%
\tau_{k-1}\nu^{-k}\right)  .
\]

\end{corollary}

\begin{proof}
First assume that $\nu\neq0.$ By (\ref{As3}) we have
\[
u_{n}=\frac{u_{1}-\nu u_{0}}{\lambda-\nu}\lambda^{n}+\frac{\lambda u_{0}%
-u_{1}}{\lambda-\nu}\nu^{n}+\frac{\lambda^{n}}{\lambda-\nu}\sum_{k=1}^{n}%
\tau_{k-1}\lambda^{-k}-\frac{\nu^{n}}{\lambda-\nu}\sum_{k=1}^{n}\tau_{k-1}%
\nu^{-k}.
\]
We observe that%
\[
\sum_{k=1}^{n}\tau_{k-1}\lambda^{-k}=\sum_{k=1}^{\infty}\tau_{k-1}\lambda
^{-k}-\sum_{k=n+1}^{\infty}\tau_{k-1}\lambda^{-k}=\sum_{k=1}^{\infty}%
\tau_{k-1}\lambda^{-k}+O\left(  \lambda^{-n}\right)  ,
\]
and the same equality holds with $\lambda$ replaced by $\nu$ if $\left\vert
\nu\right\vert >1,$ which proves (\ref{un3}). On the other hand, if
$\left\vert \nu\right\vert <1,$
\begin{equation}
\left\vert \frac{\nu^{n}}{\lambda-\nu}\sum_{k=1}^{n}\tau_{k-1}\nu
^{-k}\right\vert \leq\frac{M\left(  1-\left\vert \nu\right\vert ^{n}\right)
}{\left\vert \lambda-\nu\right\vert \left(  1-\left\vert \nu\right\vert
\right)  }=O(1), \label{Asym3}%
\end{equation}
where $M=\max_{k\in\mathbb{N}}\left\vert \theta_{k}\right\vert $, which proves
(\ref{un1}). Finally, if $\left\vert \nu\right\vert =1,$%
\begin{equation}
\left\vert \frac{\nu^{n}}{\lambda-\nu}\sum_{k=1}^{n}\tau_{k-1}\nu
^{-k}\right\vert \leq\frac{Mn}{\left\vert \lambda-\nu\right\vert }=O(n),
\label{Asym4}%
\end{equation}
which proves (\ref{un2}). When $\nu=0$ one argues the same way by using
(\ref{As4}) in place of (\ref{As3}).
\end{proof}

Now we can give an asymptotic expansion of the sequences $x_{n}$ defined by
(\ref{Init}) and (\ref{RecX}).

\begin{corollary}
\label{Cor3}Let $x_{n}$ be defined by (\ref{RecX}). Define%
\begin{align*}
\lambda &  =\frac{1}{2}\left(  r+2+\sqrt{\left(  r+2\right)  ^{2}+4q}\right)
,\\
\nu &  =\frac{1}{2}\left(  r+2-\sqrt{\left(  r+2\right)  ^{2}+4q}\right)  .
\end{align*}
Then $\lambda\geq1+\sqrt{q+1},$ $\left\vert \nu\right\vert <\lambda$ and%
\[
\left\{
\begin{array}
[c]{l}%
\log x_{n}=C\lambda^{n}+O\left(  1\right)  \quad\text{if}\quad q<r+3,\\
\log x_{n}=C\lambda^{n}+O\left(  n\right)  \quad\text{if}\quad q=r+3,\\
\log x_{n}=C\lambda^{n}+D\nu^{n}+O\left(  1\right)  \quad\text{if}\quad q>r+3,
\end{array}
\right.
\]
where $C$ and $D$ are constants.
\end{corollary}

\begin{proof}
By (\ref{Pol1}) we have since $x_{n}\leq x_{n+1}$%
\[
x_{n}Q\left(  x_{n},x_{n+1}\right)  +\theta^2y_n=\beta_{q,r}x_{n}%
^{q+1}x_{n+1}^{r}\left(  1+h_{n}\right)  ,\quad h_{n}=O\left(  \theta^2y_nx_{n}^{-1}\right)  .
\]
Taking the logarithms in (\ref{RecX}) yields%
\[
\log x_{n+2}-\left(  r+2\right)  \log x_{n+1}-q\log x_{n}=\log\beta_{q,r}%
+\log\left(  1+h_{n}\right)  .
\]
With the notations of Corollary \ref{Corol}, define $u_{n}=\log x_{n},$
$A=r+2,$ $B=-q$ and $\tau_{n}=\log\beta_{q,r}+\log\left(  1+h_{n}\right)  .$
Then $\tau_{n}$ is bounded by (\ref{Min}) and
\begin{align*}
\lambda &  =\frac{1}{2}\left(  r+2+\sqrt{\left(  r+2\right)  ^{2}+4q}\right)
\geq1+\sqrt{q+1},\\
\nu &  =\frac{1}{2}\left(  r+2-\sqrt{\left(  r+2\right)  ^{2}+4q}\right)
\leq0.
\end{align*}
Hence $\lambda\geq2$ and $q<\lambda^{2}$, which implies%
\[
\left\vert \nu\right\vert =\frac{q}{\lambda}<\lambda.
\]
Moreover
\[
\left\vert \nu\right\vert <1\Leftrightarrow\sqrt{\left(  r+2\right)  ^{2}%
+4q}<r+4\Leftrightarrow q<r+3.
\]
Therefore Corollary \ref{Corol} applies, which proves Corollary \ref{Cor3}.
\end{proof}


\begin{thebibliography}{9}                                                                                                %


\bibitem {DS}D. Duverney and I. Shiokawa, \textit{Irrationality exponents of
numbers related with Cahen's constant}, to appear in Monatsh. Math. 

\bibitem {DKS}D. Duverney, T. Kurosawa and I. Shiokawa, \textit{Transformation
formulas of finite sums into continued fractions}, \url{http://arxiv.org/abs/1912.12565}

\bibitem {H1}N.W. Hone, \textit{Curious continued fractions, non linear
recurrences, and transcendental numbers}, J. Integer Seq. 18 (2015), Article 15.8.4.

\bibitem {H2}A.N.W. Hone, \textit{Continued fractions for some transcendental
numbers, }Monatsh. Math. 174 (2015), 1-7.

\bibitem {H3}A.N.W. Hone, \textit{On the continued fraction expansion of
certain Engel series, }J. Number Theory 164 (2016), 268-281.

\bibitem {H4}A. N. W. Hone and J. L. Varona, \textit{Continued fractions and
irrationality exponents for modified Engel and Pierce series}, Monatsh. Math.
190 (2019), 501-516.

\bibitem {Roth}K. F. Roth, \textit{Rational approximations to algebraic
numbers}, Mathematika, vol. 2, n$%
%TCIMACRO{\U{b0}}%
%BeginExpansion
{{}^\circ}%
%EndExpansion
1,$ 1955, 1-20.

\bibitem {Va}J. L. Varona, \textit{The continued fraction expansion of certain
Pierce series}, J. Number Theory 180 (2017), 573-578.
\end{thebibliography}
\end{document}